\documentclass[oneside,A4paper,draft]{article}
\usepackage{amsmath, amssymb, latexsym, amscd, amsthm}
\usepackage[]{hyperref}
\usepackage[dvips]{graphicx}
\usepackage{graphicx}
\usepackage[top=2.5cm, bottom=2.5cm, left=2.5cm, right=2.5cm]{geometry}

\usepackage{mathptmx}

\newcommand{\R}{\mathbb R}
\newcommand{\N}{\mathbb N}

\newcommand{\X}{\mathfrak X}
\newcommand{\G}{\mathcal G}

\newcommand{\F}{\mathcal F}

 \newtheorem{theo}{Theorem}[section]
 \newtheorem{coro}[theo]{Corollary}
 \newtheorem{lemm}[theo]{Lemma}
 \newtheorem{prop}[theo]{Proposition}

\numberwithin{equation}{section}

\numberwithin{equation}{subsection}

\date{}

\begin{document}

\begin{center}
\textbf{{\LARGE An approach for metric space with a convex combination operation and applications}}
\vspace{20pt}

{\bf Nguyen Tran Thuan}\footnote[1]{Email: thuan.nguyen@vinhuni.edu.vn. Tel.: +84913099395}\\
\emph{Department of Mathematics, Vinh University, Nghe An Province, Vietnam}
\end{center}

{\bf Abstract.}
In this paper, we embed metric space endowed with a convex combination operation, named convex combination space, into a Banach space and the embedding preserves the structures of metric and convex combination. For random element taking values in this kind of space, applications of embedding are also established. On the one hand, some nice properties of expectation such as representation of expected value through continuous affine mappings, the linearity of expectation will be given. On the other hand, the notion of conditional expectation will be also introduced and discussed. Thanks to embedding theorem, we establish some basic properties of conditional expectation, Jensen's inequality, convergences of martingales and ergodic theorem.


\section{Introduction} 


Probability theory in linear spaces has long been considered and extended to more general models which are nonlinear, such as hyperspaces of linear space or metric spaces generally. Basic objects such as expectation, conditional expectation of random element taking values in metric space also have attracted attention of many researchers. Probably the first
author introduced a concept of mathematical expectation of a random element with values in a metric space was Doss \cite{Do} in 1949. After this paper, other authors gave many different definitions of expectation and conditional expectation in different kinds of metric spaces via various ways.  We can mention the works of \'{E}mery and Mokobodzki \cite{EM}, Herer \cite{He1, He2, He3}, Raynaud de Fitte \cite{Fi}, Sturm \cite{CS, St}, or the monograph of Molchanov \cite{Mo}. 

In 2006, Ter\'{a}n and Molchanov \cite{TM} introduced the concept of  convex combination space and the class of these spaces is  larger than not only the class of Banach spaces but also the class of hyperspace of compact subsets, as well as the class of upper semicontinuous functions (also called fuzzy sets) with compact support in Banach space \cite{TM}. Besides, the authors also provided many interesting illustrative examples of this concept, e.g., the space of all cumulative distribution functions or the space of upper semicontinuous functions with $t$-norm.  Convex combination space is a metric space endowed with a convex combination operation and the extension from
linear space to convex combination space is not trivial. Some very basic sets, such as singletons and balls, may fail to be convex in convex combination space. This may not match with usual intuition but occurs in many practical models. For example, consider the hyperspace of all compact subsets of Banach space with the convex combinations being generated by the Minkovski addition and scalar multiplication. Then $\lambda A +(1-\lambda)A$ does not equal to $A$ unless $A$ is convex, it means that $A$ is non-convex singleton in such a space. Another example is the space of integrable probability distributions, where the convex combinations is generated by the convolution operation (see \cite{TM, Te}). For random element taking values in convex combination space, its expected value was constructed by Ter\'{a}n and Molchanov. This notion of expectation extended the corresponding one when considering not only in Banach space but also in hyperspace of compact subsets. Furthermore, the authors also established the Etemadi strong law of large numbers (SLLN) for normalized sums of pairwise independent, identically distributed (i.i.d.) random elements  in this kind of space (\cite{TM}, Theorem 5.1), other applications can be found in \cite{QT, Te, TQN}.

Although convex combination space may have many singletons being not convex, it always contains a  subspace (we will call convexifiable domain) in which every singletons and balls are convex, moreover the authors in \cite{TM} shown that this subspace has some properties resembling linearity. Therefore, it is natural to ask whether this convexifiable domain can be embedded isometrically into some normed linear space such that the structure of convex combination is preserved. A worth note is that the expectation of every integrable random element taking values in convex combination space always belongs to this convexifiable domain. Therefore, if embedding is established, we will have more tools to explore this type of expectation as well as properties of convex combination space.

In this paper, we will answer the question mentioned above. Namely, we will show that the convexifiable domain of a complete convex combination space can be embedded into a Banach space such that the embedding is isometric and  the structure of convex combination is preserved, this will be presented in Section 3. 

Main applications of the approach via embedding theorem will be presented in Section 4. On the one hand, some nice properties of expectation including both representation of expected value through continuous affine mappings and Jensen's inequality (was proved first by Ter\'{a}n \cite{Te} and will be proved again in this work in another way) will be given. On the other hand, the notion of conditional expectation of integrable random element taking values in convex combination space will be also introduced and discussed. Thanks to embedding theorem, we establish some basic properties of conditional expectation, Jensen's inequality, convergences of martingales and ergodic theorem.

Finally, some miscellaneous applications and remarks will be discussed in Section 5.

\section{Preliminaries}

For the reader's convenience, we now present a short introduction to the approach
given by Ter\'an and Molchanov in \cite{TM}. Let
$(\X,d)$ be a metric space, for $u, x\in \X$, we denote $\|x\|_u:=d(u, x)$. Based on
$\X$, introduce a \emph{convex combination operation}, which
for all $n\geqslant 2$, numbers $\lambda_1,\ldots,\lambda_n>0$ that satisfy
$\sum_{i=1}^n\lambda_i=1$, and all $u_1,\ldots, u_n\in\X$,
this operation produces an element of $\X$, which is denoted by $[\lambda_1,u_1;\ldots;\lambda_n,u_n]$
 or $[\lambda_i,u_i]_{i=1}^n$. Note that $[\lambda_1,u_1;\ldots;\lambda_n,u_n]$ and the shorthand $[\lambda_i,u_i]_{i=1}^n$
have the same intuitive meaning as the more familiar $\lambda_1u_1+\cdots+\lambda_nu_n$
and $\sum_{i=1}^n \lambda_iu_i$, but $\X$ is not assumed to have any addition or multiplication. Suppose that $[1,u]=u$ for every $u\in\X$ and that the following
properties are satisfied:\\
(CC.i) (Commutativity) $[\lambda_i,u_i]_{i=1}^n=[\lambda_{\sigma(i)},u_{\sigma(i)}]_{i=1}^n$
for every permutation $\sigma$ of $\{1,\ldots,n\}$;\\
(CC.ii) (Associativity) $[\lambda_i,u_i]_{i=1}^{n+2}=\big[\lambda_1,u_1;\ldots;\lambda_n,u_n;\lambda_{n+1}+\lambda_{n+2},
\big[\frac{\lambda_{n+j}}{\lambda_{n+1}+\lambda_{n+2}},u_{n+j}\big]_{j=1}^2\big];$\\
(CC.iii) (Continuity) if $u,v\in\X$ and
$\lambda^{(k)}\rightarrow\lambda\in (0;1)$ as $k\rightarrow\infty$,
then 
$[\lambda^{(k)},u;1-\lambda^{(k)},v]\rightarrow
[\lambda,u;1-\lambda,v]$;\\
(CC.iv) (Negative curvature) if $u_1,u_2,v_1,v_2\in\X$ and
$\lambda\in (0,1)$, then
$$d([\lambda,u_1;1-\lambda,u_2],[\lambda,v_1;1-\lambda,v_2])\leqslant\lambda
d(u_1,v_1)+(1-\lambda) d(u_2,v_2);$$
Based on the inductive method and (CC.ii), this axiom can be extended to convex combinations of $n$ elements, as follows: if $u_i, v_i \in \X$, $\lambda_i \in (0;1)$ with $\sum_{i=1}^n \lambda_i =1$, then $d([\lambda_i, u_i]_{i=1}^n, [\lambda_i, v_i]_{i=1}^n)\leqslant \sum_{i=1}^n \lambda_i d(u_i, v_i).$\\
(CC.v) (Convexification) for each $u\in\X$, there exists
$\lim_{n\rightarrow\infty}[n^{-1},u]_{i=1}^n$, which will be
denoted by $K_\X u$ (or $Ku$ when no confusion can
arise), and
$K$ is called the \emph{convexification operator}.\\
Then, the metric space $(\X, d)$ endowed with a convex
combination operation is referred to as the  \emph{convex combination
space} (CC space for short) and we denote $(\X, d, [.,.])$ or $\X$ shortly. We can find from axiom (CC.v) that $[n^{-1}, u]$ is different from $u$ in general, so $Ku$ and $u$ may be not identical. If $Ku=u$, then $u$ will be called \emph{convex point} of $\X$, subspace $K(\X)$ will called \textit{convexifiable domain}. If $K(\X)=\X$ then $\X$ is said to be \emph{convexifiable} and then $[.,.]$ will be called \textit{unbiased} convex combination operation. Conditions (CC.i)--(CC.v) above imply the following properties:

(2.1) (\cite{TM}, Lemma 2.1) For every $u_{11},\ldots,u_{mn}\in\X$ and
$\alpha_1,\ldots,\alpha_m,\beta_1,\ldots,\beta_n>0$ with
$\sum_{i=1}^m\alpha_i=\sum_{j=1}^n\beta_j=1$, we have
$[\alpha_i,[\beta_j,u_{ij}]_{j=1}^n]_{i=1}^m=[\alpha_i\beta_j,u_{ij}]_{i=1,j=1}^{i=m,j=n}.$

(2.2) (\cite{TM}, Lemma 2.2) The convex combination operation is jointly
continuous in its $2n$ arguments.

(2.3) (\cite{TM}, Proposition 3.1) The convexification operator $K$ is linear, that
is $K([\lambda_j,u_j]_{j=1}^n)=[\lambda_j,Ku_j]_{j=1}^n$.

(2.4) (\cite{TM}, Corollary 3.3) If $u\in\X$ and
$\lambda_1,\ldots,\lambda_n>0$ with $\sum_{j=1}^n\lambda_j=1$, then
$K([\lambda_j,u]_{j=1}^n)=Ku=[\lambda_j,Ku]_{j=1}^n$. Hence, $K$ is
an idempotent operator in $\X$.

(2.5) (\cite{TM}, Proposition 3.5) For $\lambda_1,\lambda_2,\lambda_3>0$ with
$\lambda_1+\lambda_2+\lambda_3=1$ and $u,v\in\X$,
$$[\lambda_1,u;\lambda_2,Kv;\lambda_3,Kv]=[\lambda_1u;(\lambda_2+\lambda_3),Kv].$$

(2.6) (\cite{TM}, Proposition 3.6) The mapping $K$ is non-expansive with respect to
metric $d$, i.e., $d(Ku,Kv)\leqslant d(u,v)$.

\noindent\textbf{Remark 1.} Let $\lambda_k \subset (0;1)$, $\lambda_k\to 0$ and $u, v\in \X$. By (CC.iv) and property (2.4), we have
\begin{align*}
d([\lambda_k, Ku ; 1-\lambda_k, Kv], Kv)=d([\lambda_k, Ku ; 1-\lambda_k, Kv], [\lambda_k, Kv ; 1-\lambda_k, Kv])\leqslant \lambda_kd(Ku, Kv)\to 0
\end{align*}
as $k \to \infty$. It follows $[\lambda_k, Ku ; 1-\lambda_k, Kv] \to Kv$ and this remark ensures to extend weights $\lambda_i$ from $(0;1)$ to $[0;1]$ for elements in $K(\X)$, it means that we can define $[\lambda_i, x_i]_{i\in I}=[\lambda_i, x_i]_{i\in J}$, where $x_i\in K(\X)$, $\sum_{i\in I} \lambda_i=\sum_{i\in J} \lambda_i =1$, $J=\{i\in I : \lambda_i>0\}$.

\begin{prop}
If $(\X, d)$ is a separable and complete CC space, then so is $(K(\X), d)$.
\end{prop}
\begin{proof}
The separability of $K(\X)$ is obvious. It follows from Proposition 3.7 in \cite{TM} that $K(\X)$ is a closed subset of complete metric space $\X$, hence  $K(\X)$ is complete.
\end{proof}


\section{Embedding theorem}

First, we need to recall the embedding for convex structure given by \'{S}wirszck \cite{Sw}. In his work, \'{S}wirszck introduced the notion of semiconvex set as follows: A \emph{semiconvex set} is a set $\mathbb S$ together with a family of binary operations $\{P_\lambda : \mathbb S \times\mathbb S \to \mathbb S, \lambda \in (0;1)\}$ satisfying the following axioms: For $x, y, z \in \mathbb S$ and $\lambda, \mu \in (0;1)$, (S.i) (Reflexivity) $P_\lambda(x, x)=x$; (S.ii) (Symmetry) $P_\lambda(x, y)=P_{(1-\lambda)}(y, x)$; (S.iii) (Associativity) $P_r(P_\lambda(x, y), z)=P_{r\lambda}(x, P_\mu(y,z))$ for $r=\mu/(1-\lambda+\lambda \mu)$. Sometimes for completeness, we also include the binary identity functions $P_1$ and $P_0$ defined as $P_1(x, y)=x$ and $P_0(x, y)=y$. Then (S.i) and (S.ii) hold for $\lambda \in [0;1]$, and (S.iii) holds with $\lambda(1-\mu)\neq 1$. Also in \cite{Sw}, the author also shown that a semiconvex set $\mathbb S$ may be embedded as a convex subset of a vector space if and only if it satisfies cancellation law, that is, $P_r(x, y)=P_r(x, z)$ for any $x, y, z \in \mathbb S$, $r\in (0;1)$ implies that $y=z$. Therefore, if the cancellation law in $\mathbb S$ holds, then there exist a vector space $(\mathbb V, +, .)$ and an  one-to-one correspondence $\rho: \mathbb S \to \rho(\mathbb S)=\mathbb U\subset \mathbb V$ such that $\rho(P_\lambda(x, y))=\lambda\rho(x) + (1-\lambda) \rho(y)$ for all $x, y \in \mathbb S$, $\lambda \in [0;1]$. For more details, the readers can refer to \cite{Fl, Sw}.

\begin{prop}
If $(\X, d, [.,.])$ is a CC space, then $K(\X)$ is a semiconvex set with $P_\lambda(x, y)=[\lambda,x\,;1-\lambda,y]$, $x, y\in K(\X)$, $\lambda \in [0;1]$.
\end{prop}
\begin{proof}
It is easy to see that the axioms (S.i), (S.ii) and (S.iii)  are implied by property (2.4), (CC.i) and (CC.ii) respectively. 
\end{proof}

The following proposition establishes a metric cancellation law in $K(\X)$ and it plays the key role in obtaining the embedding theorem.
\begin{prop} \emph{(Metric cancellation law)} Let
$\X$ is a CC space and  $x, y, z \in K(\X)$, $\lambda\in [0; 1]$. Then,
\begin{align*}
d([\lambda, x ; 1-\lambda , y], [\lambda, x ; 1-\lambda , z])=(1-\lambda)d(y, z).
\end{align*}
In particular, the algebraic cancellation law holds, i.e., if $[\lambda, x\,;1-\lambda , y] = [\lambda, x \,; 1-\lambda , z]$ for some $\lambda \in [0; 1)$, then $y=z$.
\end{prop}
\begin{proof} If $\lambda=0$ or $\lambda =1$, then the conclusion is trivial. We now consider $\lambda \in (0;1)$.

\textit{Step 1.} - The first auxiliary result: If $\lambda_k \subset (0;1)$ and $\lambda_k\to 0$, then $[\lambda_k, u ; 1-\lambda_k, Kv]\to Kv$ as $k \to \infty$ for $u, v\in \X$. It is easy to see due to
\begin{align*}
d([\lambda_k, u ; 1-\lambda_k, Kv], Kv)=d([\lambda_k, u ; 1-\lambda_k, Kv], [\lambda_k, Kv ; 1-\lambda_k, Kv])\leqslant \lambda_kd(u, Kv)\to 0
\end{align*}
as $k \to \infty$.\\
- The second auxiliary result: If $u, v \in K(\X)$ then $d([\lambda, u ; 1-\lambda, v], u)=(1-\lambda)d(u, v)$ and $d([\lambda, u ; 1-\lambda, v], v)=\lambda d(u, v)$. Indeed, by (CC.iv) and (2.5)
\begin{align*}
&d([\lambda, u ; 1-\lambda, v], u)=d([\lambda, u ; 1-\lambda, v], [\lambda, u ; 1-\lambda, u])\leqslant (1-\lambda)d(u, v)\\
& d([\lambda, u ; 1-\lambda, v], v)=d([\lambda, u ; 1-\lambda, v], [\lambda, v ; 1-\lambda, v])\leqslant\lambda d(u, v)
\end{align*}
and by triangular inequality,
$$d(u, v)\leqslant d([\lambda, u ; 1-\lambda, v], u) + d([\lambda, u ; 1-\lambda, v], v)\leqslant (1-\lambda)d(u, v) + \lambda d(u, v) =d(u, v).$$
Thus, $d([\lambda, u ; 1-\lambda, v], u)=(1-\lambda)d(u, v)$ and $d([\lambda, u ; 1-\lambda, v], v)=\lambda d(u, v)$.

\textit{Step 2.} We denote by $m(x, y)=[1/2, x ; 1/2, y]$ the midpoint of $x, y$ and it is easy to see that $m(x, y)$ also belongs to $K(\X)$. By (CC.iv) we have
\begin{align*}
d(m(x, y), m(x, z))=d([1/2, x ; 1/2, y], [1/2, x ; 1/2, z])\leqslant 2^{-1}d(y, z).
\end{align*}
A set of four ordered points $(x, y, z, t)$ is called \textit{parallelogram} (according to this order) if $m(x,z)=m(y,t)$. In this step, we will prove that if $(x, y, z, t)$ is a parallelogram then $d(x, y)=d(t, z)$. Without loss of generality, assume that $d(x, y)\geqslant d(t, z)$. Now it is sufficient to prove that $d(x, y)\leqslant d(t, z)$. Putting $m(x, z)=m(y, t)=m_1$, we have
\begin{align*}
&d(m_1, m(y, z))=d(m(y, t), m(y, z))\leqslant 2^{-1}d(t, z),\\
&d(m_1, m(y, z))=d(m(z, x), m(z, y))\leqslant 2^{-1}d(x, y) \tag{3.1}.
\end{align*}
Moreover, \begin{align*}
m(m(x, t), m(y, z))&=[1/2, [1/2, x ; 1/2, t]; 1/2, [1/2, y; 1/2, z]]=[1/4,x ; 1/4, y; 1/4, z ; 1/4, t]\\
&=[1/2, [1/2, x ; 1/2, z] ; 1/2, [1/2, y ; 1/2, t]]=[1/2, m_1 ; 1/2, m_1]=m_1,
\end{align*}
it means that $m_1$ is also the midpoint of $m(x, t)$ and $m(y, z)$. Thus $d(m_1, m(y, z))=2^{-1}d(m(x, t), m(y, z))$ by Step 1. Combining with (3.1) we obtain
\begin{align*}
d(m(x, t), m(y, z))\leqslant d(t, z)\;\mbox{ and }\;d(m(x, t), m(y, z))\leqslant d(x, y).\tag{3.2}
\end{align*}
On the other hand, 
\begin{align*}
m(x, m(y, z))&=[1/2, x; 1/2, [1/2, y ; 1/2, z]]=[1/2, x ; 1/4, y ; 1/4, z]=[1/4, x ; 1/4, y ; 1/2, m_1]\\
&=[1/4, x ; 1/4, y ; 1/2,[1/2, y ; 1/2, t]]=[1/4, x ; 1/2, y ; 1/4, t]=m(y, m(x, t))
\end{align*}
and it implies that $(x, y, m(y, z), m(x, t))$ is a parallelogram. Applying (3.2), we obtain
$$d\big(m^{(2)}(x, t), m^{(2)}(y, z)\big)\leqslant d(m(x, t), m(y, z))\leqslant d(t,z)\;\mbox{ and }\;d\big(m^{(2)}(x, t), m^{(2)}(y, z)\big)\leqslant d(x, y),$$
where $m^{(2)}(x, t)=m(x, m(x, t))=[3/4, x ; 1/4, t]$, $m^{(2)}(y, z)=m(y, m(y, z))=[3/4, y ; 1/4, z]$. Continuing this process, we derive
\begin{align*}
d\big(m^{(k)}(x, t), m^{(k)}(y, z)\big)\leqslant d(t,z)\;\mbox{ and }\;d\big(m^{(k)}(x, t), m^{(k)}(y, z)\big)\leqslant d(x, y)\;\mbox{ for all } k\in \N, k\geqslant 3\tag{3.3}
\end{align*}
with $m^{(k)}(x, t)=m\big(x, m^{(k-1)}(x, t)\big)=\big[(2^k-1)/2^k, x ; 1/2^k, t\big]$, $m^{(k)}(y, z)=\big[(2^k-1)/2^k, y; 1/2^k, z\big]$. Taking $k\to \infty$ in (3.3), applying Step 1 and the continuity of metric $d$, we obtain $d(x, y)\leqslant d(t, z)$. This completes Step 2.

\textit{Step 3.} The proposition will be completed in this step. Putting $u=[\lambda, x ; 1-\lambda , y]$, $v=[\lambda, x ; 1-\lambda , z]$ and $w=[\lambda, y ; 1-\lambda, z]$, we get
\begin{align*}
&m(u, w)=[1/2, [\lambda, x ; 1-\lambda , y] ; 1/2, [\lambda, y ; 1-\lambda, z]]=[\lambda/2, x ; 1/2, y ; (1-\lambda)/2, z]\\
&m(v, y)=[1/2, [\lambda, x ; 1-\lambda , z] ; 1/2, y]=[\lambda/2, x ; 1/2, y ; (1-\lambda)/2, z].
\end{align*}
Thus, $(u, v, w, y)$  is a parallelogram and it follows from Step 2 that $d(u, v)=d(y, w)$. On the other hand, $d(y, w)=d(y, [\lambda, y ; 1-\lambda, z])=(1-\lambda)d(y, z)$ by Step 1, so $d(u, v)=(1-\lambda)d(y, z)$. The proposition is proved.
\end{proof}

\begin{theo} Let $(\X, d, [.,.])$ be a complete and convexifiable CC space. Then, there exist a Banach space $(\mathbb E, \|.\|)$  and a map $j: \X \to \mathbb E$, where $j(\X)=\mathbb F$ is a subset of $\mathbb E$ such that

(i) $\mathbb F$ is closed and convex;

(ii) $j([\lambda, x\,;1-\lambda, y])=\lambda j(x)+(1-\lambda)j(y)$ for every $x, y \in \X$, $\lambda\in [0;1];$

(iii)  $d(x, y)=\|j(x)-j(y)\|$ for all $x, y \in \X$.

Furthermore, if $\X$ is separable then $\mathbb E$ is also separable.
\begin{proof}
Since $\X$ is convexifiable, it follows from Proposition 3.1, Proposition 3.2 and the result of \'{S}wirszck \cite{Sw} mentioned above that there exist a vector space $(\mathbb V, +, .)$ and an  one-to-one correspondence $\rho: \X \to \rho(\X)=\mathbb U\subset \mathbb V$ such that $\mathbb U$ is a convex subset of $\mathbb V$ and $\rho([\lambda, x ; 1-\lambda, y])=\lambda \rho (x) + (1-\lambda) \rho(y)$ for all $x, y \in \X$, $\lambda \in [0;1]$. Thanks to translation, we can assume without loss of generality that $0:=0_{\mathbb V}\in \mathbb U$ and denote $\rho^{-1}(0)=\theta \in \X$.  This ensures that if $u$ belongs to $\mathbb U$ then $\lambda u$ also belongs to $\mathbb U$ whenever $\lambda \in [0;1]$, moreover $\lambda u =\rho ([\lambda, x ; 1-\lambda, \theta])$, where $\rho(x)=u$. The metric structure on $\mathbb U$ is induced naturally from the corresponding one on $\X$, and we also use symbol $d$ to denote the metric on $\mathbb U$. Namely, if $u=\rho(x)$, $v=\rho(y)\in \mathbb U$, then $d(u, v)=d(\rho(x), \rho(y))=d(x, y)$. Thus, if $(\X, d)$ is complete (resp. separable) then $(\mathbb U, d)$ is also complete (resp. separable). From Proposition 3.2, we have
\begin{align*}
d(\lambda u, \lambda v)=d([\lambda, x ; 1-\lambda, \theta], [\lambda, y ; 1-\lambda, \theta])=\lambda d(x, y)=\lambda d(u, v),\,\mbox{ for } \lambda \in [0;1] \mbox{ and } u, v \in \mathbb U.\tag{3.4}
\end{align*}
Let us denote by $\mathbb K = \{\lambda u : u\in \mathbb U, \lambda \geqslant 0\}$ the subset of $\mathbb V$ containing $\mathbb U$. For $x, y \in \mathbb K$, they will have form $x=\alpha u$, $y=\beta v$ with $\alpha, \beta \geqslant 0$, $u, v\in \mathbb U$, then $x+y=\alpha u + \beta v=(\alpha+\beta)\big(\frac{\alpha}{\alpha+\beta}u+\frac{\beta}{\alpha+\beta}v\big)$. It implies from the convexity of $\mathbb U$ that $\frac{\alpha}{\alpha+\beta}u+\frac{\beta}{\alpha+\beta}v\in \mathbb U$. Hence, $x+y\in \mathbb K$ and $\mathbb K$ is a convex cone of $\mathbb V$. We define the mapping $d_* : \mathbb K \times \mathbb K \to [0, \infty)$ as follows:
\begin{align*}
d_*(0,0)&=d(0,0)=0;\\
d_*(x, y)&=
d_*(\alpha u, \beta v)=(\alpha+\beta).d\Big(\frac{\alpha}{\alpha+\beta}u, \frac{\beta}{\alpha+\beta}v\Big),\mbox{ for } x=\alpha u, y=\beta v, \alpha, \beta \geqslant 0, \alpha+\beta >0, u, v \in \mathbb U.
\end{align*}
The mapping $d_*$ is well-defined,  independent of the choice of $\alpha u$ and $\beta v$. To see this, let $x=\alpha' u', y=\beta' v'$, $\alpha', \beta' \geqslant 0$, $u', v' \in \mathbb U$, then $\alpha u=\alpha' u'$, $\beta'v'=\beta v$ and using (3.4) (note that in degeneration cases $\alpha+\beta=0$ or $\alpha'+\beta'=0$, the proof is trivial),
\begin{align*}
d_*(\alpha u , \beta v)&=(\alpha+\beta).d\Big(\frac{\alpha}{\alpha+\beta}u, \frac{\beta}{\alpha+\beta}v\Big)=(\alpha+\beta+\alpha'+\beta').d\Big(\frac{\alpha}{\alpha+\beta+\alpha'+\beta'}u, \frac{\beta}{\alpha+\beta+\alpha'+\beta'}v\Big)\\
&=(\alpha+\beta+\alpha'+\beta').d\Big(\frac{\alpha'}{\alpha+\beta+\alpha'+\beta'}u', \frac{\beta'}{\alpha+\beta+\alpha'+\beta'}v'\Big)\\
&=(\alpha'+\beta').d\Big(\frac{\alpha'}{\alpha'+\beta'}u', \frac{\beta'}{\alpha'+\beta'}v'\Big)=d_*(\alpha' u' , \beta' v').
\end{align*}
It is clear that if $(x, y)\in \mathbb U \times \mathbb U$ then $d_*(x, y)=d(x, y)$, and (3.4) can be extended for $(x, y, \lambda)$ from $\mathbb U\times \mathbb U\times [0;1]$ to $\mathbb K \times \mathbb K \times [0,\infty)$ by
\begin{align*}
d_*(\lambda x, \lambda y)=d_*(\lambda \alpha u, \lambda \beta v)=\lambda(\alpha+\beta).d\Big(\frac{\alpha}{\alpha+\beta}u, \frac{\beta}{\alpha+\beta}v\Big)=\lambda d_*(\alpha u, \beta v)=\lambda d_*(x, y), \tag{3.5}
\end{align*}
for  $\lambda \geqslant 0$ and  $x, y \in \mathbb K$. We now show that $d_*$ is a metric on $\mathbb K$. Indeed, the symmetry and non-negative of $d_*$ are clear. If $d_*(x, y)=0$ then $d\big(\frac{\alpha}{\alpha+\beta}u, \frac{\beta}{\alpha+\beta}v\big)=0$ and we obtain $\frac{\alpha}{\alpha+\beta}u = \frac{\beta}{\alpha+\beta}v$. It follows $\alpha u = \beta v$ and $x=y$. Now for $x=\alpha u, y=\beta v, z=\gamma w \in \mathbb K$, $u, v, w \in \mathbb U, \alpha, \beta, \gamma \geqslant 0$, applying (3.5)
\begin{align*}
d_*(x, y)&=(\alpha+\beta+\gamma).d_*\Big(\frac{\alpha}{\alpha+\beta+\gamma} u , \frac{\beta}{\alpha+\beta+\gamma} v\Big)=(\alpha+\beta+\gamma).d\Big(\frac{\alpha}{\alpha+\beta+\gamma} u , \frac{\beta}{\alpha+\beta+\gamma} v\Big)\\
&\leqslant (\alpha+\beta+\gamma).d\Big(\frac{\alpha}{\alpha+\beta+\gamma} u , \frac{\gamma}{\alpha+\beta+\gamma} w\Big) + (\alpha+\beta+\gamma).d\Big(\frac{\gamma}{\alpha+\beta+\gamma} w , \frac{\beta}{\alpha+\beta+\gamma} v\Big)\\
&=d_*(\alpha u, \gamma w)+d_*(\gamma w, \beta v)=d_*(x,z)+d_*(z, y),
\end{align*}
we obtain the triangular inequality. On the other hand,
\begin{align*}
d_*(x+z, y+z)&=d_*(\alpha u +\gamma w, \beta v + \gamma w)\\
&=2(\alpha+\beta+\gamma).d\Big(\frac{\alpha}{2(\alpha+\beta+\gamma)} u +\frac{\gamma}{2(\alpha+\beta+\gamma)} w\,, \frac{\beta}{2(\alpha+\beta+\gamma)} v +\frac{\gamma}{2(\alpha+\beta+\gamma)} w\Big)\\
&=2(\alpha+\beta+\gamma).d\Big(\frac{\alpha}{2(\alpha+\beta+\gamma)} u\,, \frac{\beta}{2(\alpha+\beta+\gamma)} v\Big)=d_*(\alpha u, \beta v)=d_*(x, y),
\end{align*}
it means that the metric $d_*$ satisfies the cancellation law in $\mathbb K$. Recall that in degeneration cases, the proofs of triangular inequality and cancellation law are easy and we omit them. Applying  R{\aa}dstr\"{o}m's embedding theorem (\cite{Ra}, Theorem 1), there exist a real normed linear space $(\mathbb B, \|.\|)$ and a map $\widetilde{j}: \mathbb K \to \widetilde{j}(\mathbb K)=\mathbb W\subset \mathbb B$ such that: (a) $\widetilde{j}(\lambda x + \mu y)=\lambda \widetilde{j}(x)+\mu \widetilde{j}(y)$ for $x, y  \in \mathbb K$ and $\lambda, \mu \geqslant 0$; (b) $d_*(x, y)=\|\widetilde{j}(x)-\widetilde{j}(y)\|$ for all $x, y\in \mathbb K$; (c) $\mathbb W$ is a convex cone of $\mathbb B$.  Moreover, we can choose the normed linear space such that it is complete, i.e., $\mathbb B$ is a Banach space (if necessary, we denote by $\overline{\mathbb B}$ the completion of $\mathbb B$ and embed $\mathbb K$ to $\overline{\mathbb B}$). It is not hard to check that $\widetilde{j}(\mathbb U)$ is a convex subset contained in $\mathbb B$, complete under the metric induced by the
norm of $\mathbb B$. Putting $j=\widetilde{j}_\circ \rho: \X \to \mathbb B$ and $\mathbb F = j(\X)$, we find that $\mathbb F$ is a closed, convex subset of $\mathbb B$, moreover $j(\theta)=0$. Define $\mathbb E$ to be the closed linear subspace of $\mathbb B$ generated by $\mathbb F$. It is easy to check that the subspace $\mathbb E$ is a Banach space and the conclusions (i), (ii), (iii) of theorem hold. The remaining conclusion when $\X$ is separable, then $\mathbb F$ is too and this implies the separability of $\mathbb E$, so 
this observation completes the proof.
\end{proof}
\end{theo}

In 2011, Brown \cite{Br} introduced the notion of convex-like structure in metric space and it was suitably restated in \cite{CF} as follows. Let $(\X, d)$ be a complete metric space. Take $\X^{(n)}=\X\times\cdots\times\X$ to be the $n$-fold Cartesian product and Prob$_n$ the set of probability measures on the $n$-element set $\{1, 2, \ldots, n\}$ endowed with the $\ell_1$-metric $\|\mu-\nu\|=\sum_{i=1}^n|\mu(i)-\nu(i)|$. We say that $(\X, d)$ has a \textit{convex-like structure} if for every $n\in \N$ and $\mu\in$ Prob$_n$ there is given a continuous map $\gamma_\mu: \X^{(n)}\to \X$ such that

($\gamma.1$)\; $\gamma_\mu(x_1,\ldots,x_n)=\gamma_{\mu \circ \sigma}(x_{\sigma(1)},\ldots,x_{\sigma(n)})$ for every permutation $\sigma$ of $\{1,\ldots,n\}$;

($\gamma.2$)\; if $x_1=x_2$, then $\gamma_\mu(x_1,x_2,\ldots,x_n)=\gamma_\nu(x_1, x_3,\ldots, x_n)$, where $\nu\in$ Prob$_{n-1}$ is given by $\nu(1)=\mu(1)+\mu(2)$ and $\nu(j)=\mu(j+1)$, $2\leqslant j \leqslant n-1$;

($\gamma.3$)\; if $\mu(i)=1$, then $\gamma_\mu(x_1,\ldots,x_n)=x_i$;

($\gamma.4$)\; $d(\gamma_\mu(x_1,\ldots,x_n), \gamma_\mu(y_1,\ldots,y_n))\leqslant \sum_{i=1}^n\mu(i)d(x_i, y_i)$ for all $y_1, \ldots, y_n \in \X$;

($\gamma.5$)\; for all $\mu_1\in$ Prob$_n$, $\mu_2\in$ Prob$_m$, $\nu\in$ Prob$_2$, then $\gamma_\nu(\gamma_{\mu_1}(x_1,\ldots, x_n), \gamma_{\mu_2}(y_1,\ldots,y_m))=\gamma_\eta(x_1,\ldots,x_n,y_1,\ldots,y_m)$, where $\eta \in$ Prob$_{n+m}$ is given by $\eta(i)=\nu(1)\mu_1(i), 1\leqslant i\leqslant n$ and $\eta(j+n)=\nu(2)\mu_2(j), 1\leqslant j\leqslant m$.

\begin{prop}
	Let $(\X, d)$ be a complete metric space. Then, $\X$ is a convexifiable CC space if and only if $\X$ has a convex-like structure. In other words, a convexifiable CC space and metric space with a convex-like structure are identical.
	\begin{proof}
		On $\X$, when a convex-like structure and a convex combination operation determine each other by the identity
		$$\gamma_\mu(x_1,\ldots,x_n)=[\mu(1), x_1 ; \ldots; \mu(n), x_n]\;\mbox{ for } \mu \in \mbox{Prob}_n,$$
		then the axioms ($\gamma.1$) and ($\gamma.4$) are equivalent to the axioms (CC.i) and (CC.iv) respectively. 
		
		- Suppose that $\X$ is a convexifiable CC space. Then the axioms ($\gamma.2$), ($\gamma.3$), ($\gamma.5$) follow from (2.5), Remark 1, (2.1) respectively. Hence $\X$ has convex-like structure.
		
		- Suppose that $\X$ has a convex-like structure. Then, the axiom (CC.ii) follows from ($\gamma.5$); axiom (CC.v) is satisfied thanks to ($\gamma.2$) and in this case, the operation $[.,.]$ is unbiased. In order that $\X$ becomes a convexifiable CC space, it remains to check the axiom (CC.iii). Namely, for $u,v\in\X$ and
		$\lambda_k \rightarrow\lambda\in (0;1)$, we need to prove that $\gamma_{\lambda_k, 1-\lambda_k}(u,v) \to \gamma_{\lambda, 1-\lambda}(u,v)$ as $k \to \infty$, where $\gamma_{\lambda, 1-\lambda}$ is a convenient notation of $\gamma_\mu$ for $\mu\in$ Prob$_2$, $\mu(1)=\lambda, \mu(2)=1-\lambda$. For $0<\alpha\leqslant \beta<1$,
		\begin{align*}
		d(\gamma_{\alpha, 1-\alpha}(u,v), \gamma_{\beta, 1-\beta}(u,v))&=d(\gamma_\eta(u,v,v), \gamma_\eta(u,u,v))\;\;(\mbox{by }(\gamma.2) \mbox{ with } \eta(1)=\alpha, \eta(2)=\beta-\alpha, \eta(3)=1-\beta)\\
		&\leqslant (\beta-\alpha)d(u,v)\;\;(\mbox{by } (\gamma.4)).
		\end{align*}
		Changing the role of $\alpha, \beta$, we obtain $d(\gamma_{\alpha, 1-\alpha}(u,v), \gamma_{\beta, 1-\beta}(u,v))\leqslant |\beta-\alpha|d(u,v)$ for $\alpha, \beta \in (0;1)$. Applying this inequality, we have (CC.iii).
	\end{proof}
\end{prop} 

\noindent \textbf{Remark 2.} After all proofs in this paper completed, we have just been known the notion of convex-like structure by the supplying of Tobias Fritz and have been aware that a similar result to Theorem 3.3 was established before by Capraro and Fritz in \cite{CF}.  In their work, they proved that a convex-like structure is affinely and isometrically isomorphic to a closed convex subset of a Banach space (\cite{CF}, Theorem 9). Combining this result with Proposition 3.4 above, a convexifiable CC space also can be embedded into Banach space. However, the scheme for embedding in our proof is slightly different from theirs, our final goal for embedding is to apply R{\aa}dstr\"{o}m's result. To be more specific, in \cite{CF}, Theorem 9: Convex-like structure on $\X$ $\to$ establish algebraic cancellation law $\to$ embed $\X$ into vector space (by Stone's embedding) $\to$ prove the translation-invariant of metric $\to$ extend metric to affine hull and to whole vector space which becomes Banach space; while in Theorem 3.3: Convexifiable CC space $\X$ $\to$ establish metric cancellation law and as its corollary, obtain algebraic cancellation law $\to$ embed $\X$ into vector space (by \'{S}wirszck's embedding) $\to$ construct convex cone containing $\X$ and metric in this cone $\to$ embed into Banach space (by R{\aa}dstr\"{o}m's embedding). Therefore, we still present Theorem 3.3 as an independent rediscovery of Theorem 9 in \cite{CF}.

\section{Applications}
Throughout Section 4 and Section 5, $(\Omega,\F,P)$ is a
complete probability space without atoms, for $A\in\F$, the notation $I(A)$ (or $I_A$) is the indicator function of $A$.


Suppose that $(\X,d)$ is a metric space and $\G$ is a sub-$\sigma$-algebra of $\F$. A  mapping $X:\Omega\rightarrow\X$ is said to be $\G$\textit{-measurable}
if $X^{-1}(B)\in\G$ for all
$B\in\mathcal{B}(\X)$, where $\mathcal{B}(\X)$
is the Borel $\sigma$-algebra on $\X$. An $\F$-measurable mapping will be called \textit{random element} and when a
 random element $X$ takes finite values in $\X$,
it is called a \emph{simple random element}. A random element $X:\Omega\rightarrow\X$
is said to be $p$-\emph{order integrable} ($p>0$) if $d^p(u, X)$ is an integrable
real-valued random variable for some $u\in \X$ and when $p=1$, $X$ is said to be \textit{integrable} briefly. Note that this definition does not
depend on the selection of element $u$. The space (of equivalence classes) of all  $\G$-measurable, $p$-order integrable random elements in $\X$ will be denoted by
$L_\X^p(\G)$. We also use $L_\X^p$ to denote $L_\X^p(\F)$ and the metric on
$L_\X^p(\G)$ is defined by $\Delta_p(X,Y)=(Ed^p(X,Y))^{1/p}$, $p\geqslant 1$.

The \emph{distribution} $P_X$ of an $\X$-valued random element $X$ is defined by $P_X(B)=P(X^{-1}(B)),\forall B\in\mathcal{B}(\X),$ and two $\X$-valued random elements $X,Y$ are said to be \emph{identically distributed} if $P_X=P_Y$. The collection of $\X$-valued random elements $\{X_i, i\in I\}$ is said to be \emph{independent} (resp. \textit{pairwise independent}) if the collection of $\sigma$-algebras $\{\sigma(X_i), i\in I\}$ is independent (resp. pairwise independent), where $\sigma (X)=\{X^{-1}(B), B\in\mathcal{B}(\X)\}$.

Next, we recall some notions introduced by Ter\'an and Molchanov \cite{TM}. Assume that $(\X, d)$ is a separable and complete CC space. For a simple random element $X=[I_{\Omega_i}, x_i]_{i=1}^n$, the \emph{expectation} of
$X$ is defined by $EX=[P(\Omega_i),Kx_i]_{i=1}^n.$
It is easy to prove that if $X, Y$ are simple random elements, then $d(EX, EY)\leqslant Ed(X, Y).$

We fix $u_0\in K(\X)$ (by (CC.v),
$K(\X)\neq\emptyset)$ and $u_0$ will be considered as the
special element of $\X$. Since the metric space $\X$
is separable, there exists a countable dense subset $\{u_j, j\geqslant1\}$
of $\X$. For each $n\geqslant 1$, we define the mapping
$\psi_n:\X\rightarrow\X$ such that
$\psi_n(x)=u_{m_n(x)}$, where
$m_n(x)$ is the smallest $i\in\{0,\ldots,n\}$ such that $d(u_i,x)=\min_{0\leqslant j\leqslant
n}d(u_j,x)$. Then, $d(u_0, \psi_n(x)) \leqslant 2d(u_0, x)$ for all $n$ and all $x\in \X$.

 Since $\X$ is separable and complete, an integrable
$\X$-valued random element can be approximated by a
sequence of simple random elements. Namely, for $X\in L_\X^1$ then $X=\lim_{n\to \infty}\psi_n(X)$ and the \emph{expectation} of $X$
is defined by $EX=\lim_{n\rightarrow\infty}E\psi_n(X).$ By the approximation method, we also prove that if $X,Y\in L_\X^1$, then $d(EX, EY)\leqslant
Ed(X,Y)$.

A set $A \subset \X$ is called \emph{convex} if $[\lambda_i, u_i]_{i=1}^n \in A$ for all $u_i \in A$ and positive numbers $\lambda_i$ that sum to 1. For $A\subset \X$, we denote as $coA$ the \emph{convex hull} of $A$, which is the smallest convex subset containing $A$, and $\overline{co}A$ is the closure of $coA$ in $\X$. Let $k(\X)$ (resp. $ck(\X)$) be the set of nonempty compact (resp. convex compact) subsets of $\X$ and denote by $D_\X$ the Hausdorff metric on $k(\X)$, that is $D_\X (A, B)=\max\{\sup_{a\in A}\inf_{b\in B} d(a, b), \sup_{b\in B}\inf_{a\in A}d(b, a)\}$ for $A, B \in k(\X)$. It follows from  Theorem 6.2 \cite{TM} that if $\X$ is a separable complete CC space, then the space $k(\X)$ with the convex combination
$$[\lambda_i, A_i]_{i=1}^n=\{[\lambda_i, u_i]_{i=1}^n : u_i \in A_i, \;\text{for all}\; i\}$$
and Hausdorff metric $D_\X$ is a separable complete CC space, where the convexification operator $K_{k(\X)}$ is given by
$$K_{k(\X)}A=\overline{co}K_{\X}(A)=\overline{co}\{K_{\X}u : u\in A\}.$$
This is a nice feature of CC space. Based on this property, if a result holds for elements  in CC space then it can be uplifted to the space of nonempty compact subsets. In addition, $K_{k(\X)}(k(\X))=ck(K_\X(\X))$ by Proposition 5.1 in next section. Further details can be found in \cite{TM}.

From now until the end of paper,  we always assume that $(\X, d)$ is a separable and complete CC space. Proposition 2.1 implies that $(K(\X), d)$  is also separable, complete and convexifiable CC space. Therefore, it follows from Theorem 3.3 that $K(\X)$ can be embedded isometrically as a closed, convex subset of separable Banach space $\mathbb E$ via mapping $j$. Moreover, if $X$ is an integrable $\X$-valued random element, then $KX$ is an integrable $K(\X)$-valued random element.
\subsection{On some properties of expectation}
\begin{theo}
Let $X$ be an integrable $\X$-valued random element. Then, $j(E X)=j(E(KX))= E j(KX)$ where $j: K(\X) \to \mathbb E$ is the mapping mentioned in Theorem 3.3 and $Ej(KX)$ is the Bochner integral of $j(KX)$. In particular, if $X$ is an integrable $K(\X)$-valued random element, then $j(EX)=Ej(X)$.
\end{theo}
\begin{proof} First, observe that $j(KX)$ is a Borel-measurable random element in separable Banach space $\mathbb E$ and $E\|j(KX)\|=Ed(\theta, KX)\leqslant Ed(\theta, X)<\infty$, where the element $\theta$ was mentioned in proof of Theorem 3.3. This remark ensures for the existence of Bochner integral of $j(KX)$.
Next, Lemma 3.3 in \cite{Te} implies that $EX=E(KX)$, hence it is sufficient to prove  $j(E(KX))=Ej(KX)$. It will be done via using the technique of approximation by simple random elements. If $X$ is simple, i.e., $X=[I_{\Omega_i}, x_i]_{i=1}^n$, then
\begin{align*}
j(E(KX))=j([P(\Omega_i), Kx_i]_{i=1}^n)=\sum_{i=1}^n P(\Omega_i) j(Kx_i)=Ej(KX).
\end{align*}
In general case $X\in L_\X^1$, there exists a sequence of simple random elements $\{X_n=\psi_n(X)\}_{n\geqslant 1}$ such that $Ed(X_n, X)\to 0$ and $EX_n\to EX$ as $n\to \infty$. Since the convexification operator $K$ is non-expansive with respect to
metric $d$, we have $d(E(KX_n),E(KX))\leqslant Ed(KX_n, KX)\leqslant Ed(X_n, X)\to 0$. On the other hand, the continuity of mappings $j$ and $K$ follows that $j(KX_n)\to j(KX)$, moreover $$\|j(KX_n)\|=d(KX_n, \theta)\leqslant d(X_n, \theta)\leqslant d(X_n, u_0)+d(u_0, \theta)\leqslant 2d(X, u_0)+d(u_0, \theta)\in L_\R^1.$$ 
Applying the Lebesgue dominated convergence theorem in $\R$ and combining with the case above, we obtain
\begin{align*}
j(E(KX))=j\big(\lim_{n\to \infty} E(KX_n)\big)=\lim_{n\to \infty} j(E(KX_n))=\lim_{n\to \infty} Ej(KX_n)=Ej(KX).
\end{align*}
The proof is completed.
\end{proof}
By Theorem 4.1, we immediately derive the following corollary.
\begin{coro}
1) For $X_i\in L^1_{\X}$, we have $E([\lambda_1, X_1 ; \lambda_2, X_2])=[\lambda_1, EX_1 ; \lambda_2, EX_2]$.\\
2) Suppose that $X\in  L^1_{\X}$ and $\xi$ be a real-valued random variable such that $0< \xi < 1 $ a.s. If $\xi$ and $X$ are independent, then $E([\xi, X; 1-\xi, u])=[E\xi, EX; 1-E\xi, Ku]$, $u\in \X$.\\
3) Let $\xi$ be a real-valued random variable such that $0< \xi < 1 $ a.s. Then
 $E([\xi, u; 1-\xi, v])=[E\xi, Ku; 1-E\xi, Kv]$
for all $u, v \in \X$.
\end{coro}

\begin{proof} Applying Theorem 4.1 and property (2.3), we have
\begin{align*}
j(E([\lambda_1, X_1 ; \lambda_2, X_2]))&=Ej([\lambda_1, KX_1 ; \lambda_2, KX_2])=E(\lambda_1 j(KX_1)+\lambda_2j(KX_2))\\
&=\lambda_1 j(EX_1)+\lambda_2 j(EX_2)=j([\lambda_1, EX_1 ; \lambda_2, EX_2]).\\
j(E([\xi, X; 1-\xi, u]))&=Ej([\xi, KX ; 1-\xi, Ku])=E(\xi. j(KX))+(1-E\xi)j(Ku)\\
&=E\xi.Ej(KX)+(1-E\xi)j(Ku)=j([E\xi, EX; 1-E\xi, Ku]).\\
j(E([\xi, u; 1-\xi, v]))&=E(\xi. j(Ku)+(1-\xi)j(Kv))=j([E\xi, Ku; 1-E\xi, Kv]).
\end{align*}
The proof is completed by the injection of $j$. Note that the conclusions in  this corollary can be proved directly by using the technique of approximation by simple random elements.
\end{proof}

Consider a mapping $\varphi: \X \to \R$, it will be called \emph{convex} if $\varphi([\lambda_i, x_i]_{i=1}^n)\leqslant\sum_{i=1}^n\lambda_i \varphi(x_i),$
for all $x_1,\ldots, x_n \in \X$, $\lambda_1,\ldots, \lambda_n \in (0;1), \sum_{i=1}^n\lambda_i=1$; It will be called \emph{midpoint convex} if $\varphi([1/2, x ; 1/2, y])\leqslant (\varphi(x)+\varphi(y))/2$ for every $x, y \in \X$; It will be called \emph{lower semicontinuous} if $\varphi(x)\leqslant \liminf_n \varphi(x_n)$ whenever $x_n\to x$; It will be called \emph{affine} if both $\varphi$ and $-\varphi$ are convex. If $\X$ is convexifiable, then the notions of convex and affine can be extended for weights $\lambda_1,\ldots, \lambda_n \in [0;1]$. It is easy to see that if $f$ is affine, then so is $f+c$ for every $c\in \R$. Denote by $\X'$ the set of all continuous affine mappings $f: \X \to \R$. 

\begin{lemm}
If $\X$ is convexifiable and $\X$ has more than one element, then $\X'$ separates points of $\X$. In other words, if $f(x)=f(y)$ for all $f\in \X'$, then $x=y$.
\end{lemm}
\begin{proof}
Assume that there exist two elements $x, y\in \X$ and $x\neq y$  such that $f(x)=f(y)$ for all $f\in \X'$. Let $(\mathbb E, \|.\|)$ be the Banach space with dual $\mathbb E^*$ and $j: \X\to \mathbb E\supset \mathbb F=j(\X)$ is the mapping as in Theorem 3.3. Since $f$ is affine on $\X$, $\widetilde{f}=f_\circ j^{-1}$ is also affine on $\mathbb F$, where $j^{-1}: \mathbb F \to \X$ is inverse mapping of $j$.  We denote $\widetilde{\X}'=\{\widetilde{f}=f_\circ j^{-1}: \mathbb F \to \R, f\in \X'\}$ and $\mathbb F^*=\{g|_{\mathbb F}: \mathbb F \to \R, g|_{\mathbb F} \mbox{ is restriction of } g\in \mathbb E^* \mbox{ on }  \mathbb F\}$. It is easy to see that $\mathbb F^* \subset \widetilde{\X}'$ and $\X'\stackrel{\kappa}{=}\widetilde{\X}'$ (the notation $A\stackrel{\kappa}{=}B$ means that there exists an one-to-one correspondence $\kappa: A\to B$). It follows from $x\neq y$ that $j(x)\neq j(y)$ and by the Hahn-Banach separation theorem, there exists $h\in \mathbb E^*$ such that $h(j(x))\neq h(j(y))$. Moreover, since $j(x), j(y) \in \mathbb F$, we have $h|_{\mathbb F}(j(x))\neq h|_{\mathbb F}(j(y))$. Choosing $\overline{f}=(h|_{\mathbb F})_\circ j$, we obtain $\overline{f}\in \X'$ and $\overline{f}(x)\neq \overline{f}(y)$, this is the contradiction. It implies $x=y$, so $\X'$ separates points of $\X$.
\end{proof}

\noindent \textbf{Remark 3.} If $\X$ is not convexifiable, then $\X'$ does not separate points of $\X$ in general. Indeed, let $(\X,\|.\|)$ be the separable Banach space and  denote by $d$ the metric associated with norm $\|.\|$. For $r>1$, we consider the operation $^r[.,.]$ on $\X$ as follows: $^r[\lambda_i, x_i]_{i=1}^n=\sum_{i=1}^n \lambda_i^r x_i$. As shown in Example 5 in \cite{TM}, $^r[.,.]$ is the convex combination operation ($r$-th power combination) on $(\X,d)$  and the corresponding convexification operator $K_rx=0$ for all $x \in \X$. It implies that $K_r(\X)=\{0\}$ and $\X$ is not convexifiable. For $x\in \X$ and $f\in \X'$ arbitrarily, $f\big(\,^r[n^{-1}, x]_{i=1}^n\big)=\sum_{i=1}^nn^{-1}f(x)=f(x)$ for all $n$. Taking $n\to \infty$ and using the continuity of $f$, we have $f(x)=f(K_rx)=f(0)$. It means that $f$ is a constant function on $\X$, so $\X'$ contains only constant functions (moreover $\X'\stackrel{\kappa}{=}\R$). Hence, $\X'$ does not separate points of $\X$.

\begin{theo} Let $\X$ be a convexifiable CC space and $X$ be an integrable $\X$-valued random element. Then,

(i) $f(X)\in L_\R^1$ for all $f\in \X'$;

(ii) An element $m\in \X$ is the expectation of $X$ if and only if $f(m)=Ef(X)$ for all $f\in \X';$

\end{theo}
\begin{proof} Throughout this proof, we use the notations as in Theorem 3.3 and Lemma 4.3. 

(i) We will prove that for each $f\in \X'$, there exists a constant $C$ such that  $|f(x)|\leqslant C(d(\theta, x)+1)$ for all $x\in \X$. To do this, it is sufficient to prove that for each $\widetilde{f}\in \widetilde{\X}'$, $|\widetilde{f}(x)|\leqslant C(\|x\|+1)$ for all $x\in \mathbb F$. Assume to the contrary that the conclusion does not hold, then there exists a sequence $\{x_n\}_{n\geqslant 1} \subset \mathbb F$ such that $|\widetilde{f}(x_n)|>n(\|x_n\|+1)$ for all $n$. Since $0<((1+\|x_n\|)n)^{-1}\leqslant 1$ for all $n\geqslant 1$ and $0\in \mathbb F$, the convexity of $\mathbb F$ implies $\frac{x_n}{(1+\|x_n\|)n}\in \mathbb F$. We have
\begin{align*}
\widetilde{f}\Big(\frac{x_n}{(1+\|x_n\|)n}\Big)=\widetilde{f}\Big(\frac{1}{(1+\|x_n\|)n}.x_n+\Big(1-\frac{1}{(1+\|x_n\|)n}\Big).0\Big)=\frac{1}{(1+\|x_n\|)n}\widetilde{f}(x_n)+\Big(1-\frac{1}{(1+\|x_n\|)n}\Big)\widetilde{f}(0).
\end{align*}
It follows
\begin{align*}
\Big|\widetilde{f}\Big(\frac{x_n}{(1+\|x_n\|)n}\Big) - \Big(1-\frac{1}{(1+\|x_n\|)n}\Big)\widetilde{f}(0)\Big|=\frac{|\widetilde{f}(x_n)|}{(1+\|x_n\|)n}>1\;\mbox{ for all } n.\tag{4.1}
\end{align*}
Taking $n\to \infty$, the continuity of $\widetilde{f}$ implies that the LHS of (4.1) tends to 0, this is the contradiction. Therefore, $|f(X)|\leqslant C(d(\theta, X)+1)$ and this inequality implies $f(X)\in L_\R^1$.

(ii) Since $X\in L_\X^1$, the conclusion (i) ensures for the  existence of $Ef(X)$ for all $f\in \X'$. The necessity part of (ii) is easy, it can be proved through using the technique of approximation by simple random elements, so we omit the proof. We now prove the sufficiency part. Assume that $f(m)=Ef(X)$ for all $f\in \X'$, the necessity part follows that $f(m)=f(EX)$ for all $f\in \X'$. If $\X$ has one element, then $EX=m$ obviously. If $\X$ has more than one element, then applying Lemma 4.3, we obtain $m=EX$.
\end{proof}

Note that for $f\in \X'$, $$f(Kx)=f\big(\lim_{n\to\infty}[n^{-1}, x]_{i=1}^n\big)=\lim_{n\to \infty} n^{-1}\sum_{i=1}^n f(x)=f(x)$$ for all $x\in \X$. Hence, the following corollary is obtained immediately from Theorem 4.4.
\begin{coro}
	Let $\X$ be a CC space and $X$ be an integrable $\X$-valued random element. Then, $f(X)=f(KX)\in L_\R^1$ for all $f\in \X'\subset (K(\X))'$ and an element $m\in K(\X)$ is the expectation of $X$ if and only if $f(m)=Ef(KX)$ for all $f\in (K(\X))'.$
\end{coro}

\begin{prop} \emph{(\cite{Te}, Theorem 3.1)} Let $\varphi: \X \to \R$ be midpoint convex and lower semicontinuous, and let $X$ be an integrable $\X$-valued random element. Then $\varphi(EX)\leqslant E\varphi(X)$ whenever $\varphi(X)$ is integrable.
\end{prop}

\begin{proof}
This proposition established Jensen's inequality in CC space and it is a main result of Ter\'an \cite{Te}. It was proved nicely in \cite{Te} by using SLLN. Beside the approach of Ter\'an, we will present in this proof another method through combining embedding theorem and a corresponding version of Jensen's inequality in Banach space.   
First, we will prove that if $\varphi: \X\to\R$ is midpoint convex and lower semicontinuous then $\varphi(Kx)\leqslant \varphi(x)$, $x\in \X$. Indeed, since $[n^{-1}, x]_{i=1}^n \to Kx$, the subsequence $\big\{[2^{-m}, x]_{i=1}^{2^m}\big\}_{m\geqslant 1}$ also tends to $Kx$ when $m\to \infty$. Applying the first part of proof of Proposition 5.3 (will be given in next section), we have 
$$\varphi(Kx)=\varphi\big(\lim_{m\to \infty} [2^{-m}, x]_{i=1}^{2^m}\big)\leqslant \liminf_{m\to \infty} \varphi\big( [2^{-m}, x]_{i=1}^{2^m}\big)\leqslant \liminf_{m\to \infty}\,2^{-m}\sum_{i=1}^{2^m} \varphi(x)=\varphi(x).$$
This reason implies $\varphi(KX)\leqslant \varphi(X)$, in particular $\varphi^+(KX)\leqslant \varphi^+(X)$ where $\varphi^+=\max\{0, \varphi\}$. We now consider two cases as follows:

\textit{Case 1.} $\varphi(KX)$ is integrable. This implies that  $E\varphi(KX)$ is finite and $E\varphi(KX) \leqslant E\varphi(X)$. With $j^{-1}: \mathbb F \to K(\X)$, putting $\widetilde{\varphi}=\varphi_\circ j^{-1} : \mathbb F \to \mathbb R$, we derive 
$$\widetilde{\varphi}(x/2 + y/2)=\varphi\big([1/2, j^{-1}(x)\,; 1/2, j^{-1}(y)]\big)\leqslant \big(\varphi_\circ j^{-1}(x)+\varphi_\circ j^{-1}(y)\big)/2=\big(\widetilde{\varphi}(x) + \widetilde{\varphi}(y)\big)/2$$ for all $x, y \in \mathbb F$, it means that $\widetilde{\varphi}$ is midpoint convex on $\mathbb F$. Since $\varphi$ is lower semicontinuous on $\X$ and $j^{-1}$ is isometric, $\widetilde{\varphi}$ is lower semicontinuous on $\mathbb F$. Then, $\widetilde{\varphi}$ is midpoint convex as well as lower semicontinuous on $\mathbb F$, it implies that $\widetilde{\varphi}$ is convex on $\mathbb F$. Applying Jensen's inequality (\cite{Pe}, Theorem 3.10(ii)), we get  $\widetilde{\varphi}(E(j(KX)))\leqslant E \widetilde{\varphi}(j(KX))$. On the other hand, Theorem 4.1 follows that $\widetilde{\varphi}(j(E(KX)))= \widetilde{\varphi}(E (j(KX)))$, and this is equivalent to  $\varphi(E(KX))= \widetilde{\varphi}(E (j(KX)))$. Combining the arguments above, we obtain $\varphi(EX)=\varphi(E(KX))\leqslant E\varphi(KX)\leqslant E\varphi(X)$.

\textit{Case 2.}  $\varphi(KX)$ is not integrable. Putting $\varphi_n=\max\{-n, \varphi\}$, $n=1, 2, \ldots$, we have $\varphi_n \searrow \varphi$ and $\varphi_n(KX)$ is integrable for each $n$ thanks to $\varphi^+ \geqslant \varphi_n\geqslant -n$. It is not hard to check that $\{\varphi^+, \varphi_n, n\geqslant 1\}$ is also a collection of lower semicontinuous and midpoint convex functions on $\X$. According to Case 1, we obtain $\varphi_n(EX)\leqslant E\varphi_n(X)$ for all $n$. Taking $n\to \infty$ and using the monotone convergence theorem, we derive $\varphi(EX)\leqslant E\varphi(X)$. This completes the proof.
\end{proof}

\subsection{On notion of conditional expectation}
The notion of conditional  expectation of a random element taking values in concrete metric spaces was introduced by some authors via various ways. For example, Herer \cite{He3} constructed this notion in finitely compact metric space with nonnegative curvature. Sturm \cite{St} dealt with problem in global NPC space and conditional expectation was defined as a minimizer of the ``variance''. Other definitions can be found in \cite{CS, He2, Fi}. In this part, we will discuss the notion of conditional expectation in CC space $\X$ and stress that all presented results below will extend corresponding ones in Banach space. The scheme to construct this notion will be proceeded through approximation method traditionally.

Let $X\in L_\X^1$. If $X=[I_{(X=x_i)}, x_i]_{i=1}^n$ is simple, then the \emph{conditional expectation} of $X$ relative to a $\sigma$-algebra $\G\subset \F$ is defined by $E(X|\mathcal G)=[E(I_{(X=x_i)}|\G), Kx_i]_{i=1}^n$ (A). With this definition, maybe the readers naturally wonder that why we do not use another form of conditional expectation, such as $E(X|\mathcal G)=[E(I_{(X=x_i)}|\G), x_i]_{i=1}^n$ (B). This can be clarified that the definition (B) will not extend the notion of expectation when $\G=\{\emptyset, \Omega\}$, and a more profound reason is that (B) will depend on the representation of $X$ while (A) will not (see property (2.5)). Hence, the definition (A) is more suitable than (B). 

From the definition (A) above, we can prove with some simple calculations that if $X$ and $Y$ are simple random elements, then $d(E(X|\G), E(Y|\G))\leqslant E(d(X, Y)|\G)$ a.s.,
where $\G$ is some sub-$\sigma$-algebra of $\F$. We now consider the general case, let $X$ be an integrable random element, i.e., $X\in L^1_\X$, the condition expectation of $X$ is defined (up to a null set) by $E(X|\G)=\lim_{n\to \infty} E(\psi_n(X)|\G)$ a.s., where the mapping $\psi_n$ was mentioned in the first part of Section 4. Note that the limit in the RHS exists due to the completeness of $L^1_\X(\G)$. It is easy to see from the above definition that if $X\in L^1_\X$ then $E(X|\G)\in L^1_{K(\X)}(\G)$.  Moreover, by applying approximation method and the Lebesgue dominated convergence theorem for conditional expectation in $\R$, we also find $d(E(X|\G), E(Y|\G))\leqslant E(d(X, Y)|\G)$ for $X, Y \in L^1_\X$ and in particular, $\|E(X|\G)\|_a\leqslant E(\|X\|_a|\G), a\in K(\X)$.

First, we will establish the Lebesgue dominated convergence theorem for conditional expectation in CC space.
\begin{prop}
Let $X_n, X$ be integrable $\X$-valued  random elements. Assume that the following hold:

(i) $d(X_n, X) \to 0$ a.s. as $n\to \infty$,

(ii) there exist a function $f \in L^1_{\R}$ and some $a\in \X$ such that $\|X_n\|_{a} \leqslant f$ a.s. for all $n$.\\
Then $d(E(X_n|\G), E(X|\G))\to 0$ a.s. as $n\to \infty$.  
\end{prop}
\begin{proof}
By triangular inequality, $d(X_n, X)\leqslant \|X_n\|_{a}+\|X\|_{a}\leqslant f +\|X\|_{a}$ a.s. Since $\|X\|_{a} +f \in L^1_\R$, it follows from the Lebesgue dominated convergence theorem for conditional expectation in $\R$ that
$$\lim_{n\to \infty} d(E(X_n|\G), E(X|\G)) \leqslant \lim_{n\to \infty} E(d(X_n, X)|\G)=E(\lim_{n\to \infty}d(X_n, X)|\G)=0\;\mbox{ a.s.}$$
The proof is completed.
\end{proof}

\begin{theo}
Let $X$ be an integrable $\X$-valued random element. Then, $j(E(X|\G))=j(E(KX|\G))= E(j(KX)|\G)$ a.s., where $j: K(\X) \to  \mathbb E$ is the mapping presented in Theorem 3.3.
\end{theo}
\begin{proof}
As mentioned in Theorem 4.1, $j(KX)$ is a random element in $\mathbb E$ and $j(KX)\in L_{\mathbb E}^1$. Hence, there exists the conditional expectation $E(j(KX)|\G)$, moreover $E(j(KX)|\G)\in j(K(\X))$ a.s. First, if  $X=[I_{(X=x_i)}, x_i]_{i=1}^n$ is simple, then by the definition of conditional expectation and the idempotence of $K$
\begin{align*}
E(KX|\G)&=E([I_{(X=x_i)}, Kx_i]_{i=1}^n|\G)=[E(I_{(X=x_i)}|\G), KKx_i]_{i=1}^n=E(X|\G)\;\mbox{ a.s.}\\
j(E(KX|\G))&=j([E(I_{(X=x_i)}|\G), Kx_i]_{i=1}^n)=\sum_{i=1}^n E(I_{(X=x_i)}|\G) j(Kx_i)=E(j(KX)|\G)\;\mbox{ a.s.}
\end{align*}
Next, if $X\in L_\X^1$ then there exists a sequence $\{X_n, n\geqslant 1\}$ of simple random elements such that $X_n\to X$, $\|X_n\|_{u_0}\leqslant 2\|X\|_{u_0}$, $E(X_n|\G)\to E(X|\G)$ a.s. Applying Proposition 4.6, we obtain $E(KX_n|\G)\to E(KX|\G)$ a.s. Moreover, it follows from the previous case that $E(KX_n|\G)=E(X_n|\G)$ for all $n$, and the uniqueness of limit implies $E(KX|\G)=E(X|\G)$ a.s. Since $j$ is continuous,
$$j(E(KX|\G))=j\big(\lim_{n\to \infty} E(KX_n|\G)\big)=\lim_{n\to \infty} j(E(KX_n|\G))=\lim_{n\to \infty} E(j(KX_n)|\G)=E(j(KX)|\G)\;\mbox{ a.s.},$$
where the last limit holds due to Lebesgue's dominated convergence theorem for conditional expectation in Banach space. The proof is completed.
\end{proof}

It is well-known that definition of conditional expectation $E(X|\G)$ via approximate method in separable Banach space $\mathcal E$ is equivalent to the result: ``\textit{For $X\in L_{\mathcal E}^1$, then $Y=E(X|\G)$ if and only if $Y\in L_{\mathcal E}^1(\G)$ and $EXI_A=EYI_A$ for all $A\in \G$}''. The same equivalence in CC space will be established in following result and its proof is based on embedding theorem.
\begin{theo}
Let $X\in L^1_\X$ and $a\in K(\X)$. Then $Y=E(X|\G)$ if and only if $Y\in L_{K(\X)}^1(\G)$ and
$E([I_A, X ; I_{\overline{A}}, a])=E([I_A, Y; I_{\overline{A}}, a])$ for all $A\in \G$.
\end{theo}

\begin{proof}
\textit{Necessary:} If $Y=E(X|\G)$ then $Y\in L_{K(\X)}^1(\G)$ obviously. For $A\in \G$, by Theorem 4.1 and Theorem 4.8,
\begin{align*}
j(E([I_A, Y; I_{\overline{A}}, a]))&=E(I_A j(Y)+I_{\overline{A}}j(a))=E(I_A j(E(X|\G))+I_{\overline{A}}j(a))=E(I_A E(j(KX)|\G)+I_{\overline{A}}j(a))\\
&=E(E(I_A  j(KX)|\G)+I_{\overline{A}}j(a))=E(I_A  j(KX)+I_{\overline{A}}j(a))=j(E([I_A, X ; I_{\overline{A}}, a])).
\end{align*}
The injection of $j$ implies $E([I_A, X ; I_{\overline{A}}, a])=E([I_A, Y; I_{\overline{A}}, a])$.

\emph{Sufficiency:} Assume that there exists $Y\in L_{K(\X)}^1(\G)$ such that $E([I_A, X ; I_{\overline{A}}, a])=E([I_A, Y; I_{\overline{A}}, a])$ for all $A\in \G$. We now need to prove that $Y=E(X|\G)$. Observe that the conditional expectation $E(X|\G)$ exists due to $X\in L_\X^1$. By the hypothesis, we have $j(E([I_A, X ; I_{\overline{A}}, a]))=j(E([I_A, Y; I_{\overline{A}}, a]))$ for all $A\in \G$, this is equivalent to $E(I_Aj(KX))=E(I_A j(Y))$ for all $A\in \G$. It is obvious that $j(Y)$ is $\G$-measurable and integrable, so $j(Y)=E(j(KX)|\G)$. On the other hand, $E(j(KX)|\G)=j(E(X|\G))$ by Theorem 4.8. Thus $j(Y)=j(E(X|\G))$ and it follows that $Y=E(X|\G)$.
\end{proof}

The proposition below will give some basic properties of conditional expectation. The proof is easy thanks to Theorem 4.1 and Theorem 4.8.

\begin{prop} Let $X, Y \in L^1_\X$. Then, the following hold for $\omega \in \Omega$ a.s.:\\
1) $E(E(X|\G))=EX$.\\
2) If $\sigma(X)$ and $\G$ are independent then $E(X|\G)=EX$.\\
3) If $X$ is $\G$-measurable then $E(X|\G)=KX$.\\
4) If $\xi$ is a real-valued random variable with $0< \xi< 1$ and $\xi$ is $\G$-measurable, then 
$$E([\xi, X ; 1-\xi, Y]|\G)=[\xi, E(X|\G) ; 1-\xi , E(Y|\G)].$$
In particular, $E([\lambda, X ; 1-\lambda, Y] | \G)=[\lambda, E(X|\G) ; 1-\lambda, E(Y|\G)]$ for $\lambda \in (0;1)$.\\
5) If $\G_1, \G_2$ are two $\sigma$-algebras and $\G_1\subset \G_2$ then $E(E(X|\G_1)|\G_2)=E(E(X|\G_2)|\G_1)=E(X|\G_1)$.
\end{prop}

The Jensen inequality for conditional expectation in CC space will be given in the following proposition. Note here that this result does not totally extend Proposition 4.6.
\begin{prop}
Let $\varphi : \X\to \R$ be a midpoint convex and continuous function, sub-$\sigma$-algebra $\G\subset \F$ and let $X \in L^1_\X$ such that $\varphi(X)\in L^1_\R$. Then $\varphi(E(X|\G))\leqslant E(\varphi(X)|\G)$ a.s.
\end{prop} 

\begin{proof}
Combining Jensen's inequality for Banach-valued conditional expectation (e.g., see Theorem in \cite{TW}) with embedding Theorem 3.3 and using simultaneously the same scheme as in proof of Proposition 4.6, we will have the conclusion.
\end{proof}

According to Theorem 4.4(i) and Proposition 4.11, we immediately derive the following corollary.

\begin{coro}
1) If 
$X\in L^p_\X$ then $\|E(X|\G)\|_a^p\leqslant E(\|X\|_a^p|\G)$ a.s., for arbitrarily $a\in K(\X), p\geqslant 1$.\\
2) If $X\in L^1_\X$ then $f(E(X|\G))=E(f(X)|\G)=E(f(KX)|\G)$ a.s. for all $f\in \X'$.
\end{coro}

Similar to Banach space, the notion of martingale in CC space can be defined as follows: Let $\{X_n, n\geqslant 1\}\subset L_\X^1$ and $\{\F_n, n\geqslant 1\}$ be an increasing sequence of sub-$\sigma$-algebras of $\F$. The collection $\{X_n, \F_n, n\geqslant 1\}$ is said to be \textit{martingale} if $X_n$ is $\F_n$-measurable and $E(X_{n+1}|\F_n)=X_n$ a.s. for all $n\geqslant 1$. Thanks to Corollary 4.12(1), it is easy to verify that if $\{X_n, \F_n, n\geqslant 1\}$ is a martingale then $\{\|X_n\|^p_a, \F_n, n\geqslant 1\}$ is a real-valued submartingale for $a\in \X$, $p\geqslant 1$ arbitrarily. The convergence of martingales will be established in proposition below.
\begin{prop}
	(i) Let $\{\mathcal F_{n}, n\geqslant 1\}$ be an increasing sequence of sub-$\sigma$-algebras of $\F$ and let $\F_{\infty}=\sigma(\cup_{n\geqslant 1}\F_{n})$. If $X\in L_\X^p$ with some $p\geqslant 1$, then $E(X|\F_{n})\to E(X|\F_{\infty})$ a.s. and in $L_\X^p$ as $n\to \infty$.
	
(ii) Let $\{\mathcal F_{-n}, n\geqslant 1\}$ be a decreasing sequence of sub-$\sigma$-algebras of $\F$ and let $\F_{-\infty}=\cap_{n\geqslant 1}\F_{-n}$. If $X\in L_\X^p$ with some $p\geqslant 1$, then $E(X|\F_{-n})\to E(X|\F_{-\infty})$ a.s. and in $L_\X^p$ as $n\to \infty$.
\end{prop}
\begin{proof} With the hypothesis in (i) and (ii), $\{E(X_n|\F_n), \F_n, n\geqslant 1\}$ is a martingale and $\{E(X_n|\F_{-n}), \F_{-n}, n\geqslant 1\}$ is an inverse martingale respectively. Combining convergence theorems for  Banach space-valued martingales (e.g., see Pisier \cite{Pi}, Theorem 1.5 and Theorem 1.14 for conclusion (i); Theorem in \cite{Pi}, Ch.I, Section 1.5 for conclusion (ii)) with the embedding Theorem 3.3, we obtain immediately the proof.
\end{proof}

The last result in this section, we will establish a version of Birkhoff's ergodic theorem in CC space. Let $\tau: \Omega \to \Omega$ be an $\F$-measurable transformation. A transformation $\tau$ is a \textit{measure-preserving}
 or, equivalently, $P$ is said to be $\tau$-\emph{invariant measure}, if $P(\tau^{-1}(A))=P(A)$
 for all $A\in \F$. A
set $A\in \F$  satisfying $\tau^{-1}(A)=A$ is said to be $\tau$-\emph{invariant set} and the family of all $\tau$-invariant sets will constitute a sub-$\sigma$-algebra $\mathcal I_\tau$ of $\F$ . We say that
$\tau$ is an \emph{ergodic} if $\mathcal I_\tau$ is trivial, i.e., $P(A)=0$ or $P(A)=1$ whenever $A\in \mathcal I_\tau$. 
\begin{theo}
Let $\tau$
be a measure-preserving
transformation of the probability space $(\Omega, \F, P)$ and $\mathcal I_\tau$ be the $\sigma$-algebra of invariant events with respect to $\tau$. If $X\in L_\X^1$, then $[n^{-1}, X_\circ\tau^i]_{i=0}^{n-1} \to E(X|\mathcal I_\tau)$ a.s. as $n \to \infty$.
\end{theo}
\begin{proof}
	Recall that in Th\'eor\`eme 3.1 in \cite{Fi}, Raynaud de Fitte proved a version of ergodic theorem in metric space by using the technique of approximation by discrete range random elements. To prove our result, we will present here another technique via using the embedding theorem. 
Since $X$ is integrable, Theorem 3.2 in \cite{Pa} implies that for each natural number $m$, there exists a compact subset $\mathcal K_{u,m}=\mathcal K_m$ of $\X$ such that $E(d(X, u)I(X\notin \mathcal K_m))<1/m$ and without loss of generality, we can assume that $\mathcal K_m\subset \mathcal K_{m+1}$ for all $m$. For each $n, m\geqslant 1$, defining  $Y_{m,n-1}=X_\circ \tau^{n-1}$ if $X_\circ \tau^{n-1} \in \mathcal K_m$ and $Y_{m,n-1}=u$ if $X_\circ \tau^{n-1} \notin \mathcal K_m$, we have
\begin{align*}
d([n^{-1}, X_\circ\tau^i]_{i=0}^{n-1}, E(X|\mathcal I_\tau))\leqslant & d([n^{-1}, X_\circ\tau^i]_{i=0}^{n-1}, [n^{-1}, Y_{m,i}]_{i=0}^{n-1})+ d([n^{-1}, Y_{m,i}]_{i=0}^{n-1}, [n^{-1}, KY_{m,i}]_{i=0}^{n-1})\\
&+d([n^{-1}, KY_{m,i}]_{i=0}^{n-1}, [n^{-1}, KX_\circ\tau^i]_{i=0}^{n-1})+d([n^{-1}, KX_\circ\tau^i]_{i=0}^{n-1}, E(X|\mathcal I_\tau)) \tag{4.2}.
\end{align*}
We will estimate four parts in RHS of inequality (4.2) as follows. First, since $\mathcal K_m \cup \{u\}$ is compact and $Y_{m,n} \in \mathcal K_m \cup \{u\}$ for each $m$, Proposition 5.5 (will be given in next section) follows $d([n^{-1}, Y_{m,i}]_{i=0}^{n-1}, [n^{-1}, KY_{m,i}]_{i=0}^{n-1}) \to 0$ as $n\to \infty$. Second, according to properties (2.3), (2.6) and the definition of $Y_{m,n}$, we obtain 
\begin{align*}
&d([n^{-1}, KY_{m,i}]_{i=0}^{n-1}, [n^{-1}, KX_\circ\tau^i]_{i=0}^{n-1})\leqslant d([n^{-1}, Y_{m,i}]_{i=0}^{n-1}, [n^{-1}, X_\circ\tau^i]_{i=0}^{n-1})\\
&\leqslant n^{-1}\sum_{i=0}^{n-1} d(Y_{m,i}, X_\circ\tau^i)= n^{-1}\sum_{i=0}^{n-1} d(X_\circ\tau^i, u)I(X_\circ\tau^i \notin \mathcal K_m)= n^{-1}\sum_{i=0}^{n-1} (d(X, u)I(X\notin \mathcal K_m))_\circ\tau^i.
\end{align*}
For each $m$, applying the classic Birkhoff ergodic theorem for  real-valued random variable $d(X, u)I(X\notin \mathcal K_m)$, we derive 
$$n^{-1}\sum_{i=0}^{n-1} (d(X, u)I(X\notin \mathcal K_m))_\circ\tau^i \to E(d(X, u)I(X \notin \mathcal K_m)|\mathcal I_\tau)\; \mbox{ a.s. as }n \to \infty.$$ 
Next, applying Theorem 3.3
$$d([n^{-1}, KX_\circ\tau^i]_{i=0}^{n-1}, E(X|\mathcal I_\tau))=\Big\|n^{-1}
\sum_{i=0}^{n-1} j(KX_\circ\tau^i) -j(E(X|\mathcal I_\tau))\Big\|=\Big\|n^{-1}
\sum_{i=0}^{n-1} j(KX)_\circ\tau^i -E(j(KX)|\mathcal I_\tau)\Big\|\to 0$$ a.s.
as $n\to \infty$, where the convergence comes from Birkhoff's ergodic theorem for  Banach-valued random element $j(KX)$ (\cite{Pa}, Ch.VI, Theorem 9.4). Combining above arguments, we obtain
\begin{align*}
\limsup_{n\to \infty} d([n^{-1}, X_\circ\tau^i]_{i=0}^{n-1}, E(X|\mathcal I_\tau))  \leqslant 2E(d(X, u)I(X \notin \mathcal K_m)|\mathcal I_\tau) \;\mbox{ a.s. for all } m.
\end{align*}
Finally, to get the conclusion of theorem,  it is sufficient to prove that $E(d(X, u)I(X \notin \mathcal K_m)|\mathcal I_\tau) \to 0$ a.s. as $m \to \infty$. Observe that $\{E(d(X, u)I(X \notin \mathcal K_m)|\mathcal I_\tau), m\geqslant 1\}$ is a non-increasing sequence, so the almost surely convergence is equivalent to the convergence in probability. For $\varepsilon >0$ arbitrarily,
\begin{align*}
P(|E(d(X, u)I(X \notin \mathcal K_m)|\mathcal I_\tau)| >\varepsilon)\leqslant \varepsilon^{-1} E(d(X, u)I(X \notin \mathcal K_m))\leqslant \varepsilon^{-1}  m^{-1} \to 0\;\mbox{ as } m \to \infty, 
\end{align*}
and this completes the proof of theorem.
\end{proof}

\section{Miscellaneous applications and remarks}

\begin{prop}
	If $\X$ is a complete CC space, then $K_{k(\X)}(k(\X))=ck(K_\X(\X))$. So the CC space $(ck(K_\X(\X)), D_\X)$ can be embedded isometrically into a Banach space such that convex combination structure is preserved.
\end{prop}

\begin{proof}
	For $A\in K_{k(\X)}(k(\X))$,  there exists $B\in k(\X)$ such that $A=K_{k(\X)}B=\overline{co}K_\X(B)$. It follows from the continuity of $K_\X$ that $K_\X(B)\in k(K_\X(\X))$, so $\overline{co}K_\X(B)$ is a compact and convex subset of $K_\X(\X)$. It means $A=\overline{co}K_\X(B) \in ck(K_\X(\X))$, thus $K_{k(\X)}(k(\X))\subset ck(K_\X(\X))$. The inverse  implication is easy to obtain thanks to the observation that $A=\overline{co}K_\X(A)=K_{k(\X)}A$ for $A\in ck(K_\X(\X))$.
\end{proof}

Lemma 3.3 in \cite{QT} established an inequality in CC space and it is a useful tool to obtain many limit theorems (see \cite{QT, TQN}). Now by applying Theorem 3.3, this lemma may be proved more easily as follows:

\begin{prop}
\emph{(\cite{QT}, Lemma 3.3)} Let $\{a_i, b_i, 1\leqslant i\leqslant n\} \subset [0, 1]$ be a collection of nonnegative constants with $\sum_{i=1}^n a_i=\sum_{i=1}^n b_i=1$. Then $d([a_i, Kx_i]_{i=1}^n, [b_i, Kx_i]_{i=1}^n)\leqslant \sum_{i=1}^n|a_i-b_i| d(x_i, u),$
where $x_1,\ldots, x_n, u \in \X$ are arbitrary.
\end{prop}
\begin{proof} With the notations as in Theorem 3.3, we have
\begin{align*}
d([a_i, Kx_i]_{i=1}^n, [b_i, Kx_i]_{i=1}^n)&=\Big\|\sum_{i=1}^n a_i j(Kx_i)-\sum_{i=1}^n b_i j(Kx_i)\Big\|=\Big\|\sum_{i=1}^n(a_i-b_i)j(Kx_i)\Big\|\\
&=\Big\|\sum_{i=1}^n(a_i-b_i)(j(Kx_i)-j(Ku))\Big\|\leqslant \sum_{i=1}^n|a_i-b_i|\|j(Kx_i)-j(Ku)\|\\
&=\sum_{i=1}^n|a_i-b_i|d(Kx_i, Ku)\leqslant \sum_{i=1}^n|a_i-b_i|d(x_i, u),
\end{align*}
where the last estimation follows from property (2.6).
\end{proof}

\noindent\textbf{Remark 4.} The inequality 
\begin{align*}
d([a_i, x_i]_{i=1}^n, [b_i, x_i]_{i=1}^n)\leqslant \sum_{i=1}^n|a_i-b_i| d(x_i, u)\tag{5.1}
\end{align*}
does not hold  in general for $x_1,\ldots, x_n \in \X$. It will be shown via the following example:

\noindent \textbf{Example 1.} Let $(\X, \|.\|)$ be a Banach space and we consider the  operator $^2[\lambda_i, x_i]_{i=1}^n =\sum_{i=1}^n \lambda_i^2 x_i$. As shown in Example 5 of \cite{TM}, $(\X, \|.\|,\,^2[.,.])$ is a CC space. For $0 \neq x, y \in \X$, we have
\begin{align*}
d\big(\,^2[4/5, x ; 1/5, y], \, ^2[2/5, x ; 3/5, y]\big)=\|(16x/25+y/25)-(4x/25 + 9y/25)\|=\|12x/25-8y/25\|.
\end{align*}
Choosing $y=-x/2$, we get $\|12x/25-8y/25\|=16\|x\|/25$. On the other hand, $|4/5-2/5|.\|x\|+|1/5-3/5|.\|y\|=3\|x\|/5<16\|x\|/25$, so (5.1) fails with $u=0$.

The result below is the Etemadi SLLN in CC space and it was proved in \cite{TM} via approximation method by simple random elements. However, a different proof can be obtained by combining Etemadi's SLLN in Banach space (\cite{Et}, Remark 2) with embedding Theorem 3.3 and using simultaneously the same scheme as in proof of Theorem 4.14.

\begin{prop} \emph{(\cite{TM}, Theorem 5.1)} Let $\{X, X_n, n\geqslant 1\}$ be a sequence of pairwise i.i.d. $\X$-valued random  elements. Then, $[n^{-1}, X_i]_{i=1}^n \to EX$ a.s. as $n\to \infty$.
\end{prop}


The following proposition will present a special form of Jensen's equality and it plays an important role in establishing general case. This inequality can be prove easily by combining Theorem 3.3 and a corresponding version in Banach space, moreover it was proved directly by Ter\'{a}n (\cite{Te}, Lemma 3.2.). However, in proof of this inequality below, we will give another direct manner which seems to be more simple than the one of Ter\'{a}n \cite{Te}.

\begin{prop}
Let $\varphi: \X\to \R$ be a midpoint convex function and $\{x_i\}_{i=1}^n \subset K(\X)$ be a sequence of convex points of $\X$. If $\{q_i\}_{i=1}^n$ is a sequence of positive rational numbers with $\sum_{i=1}^n q_i =1$, then
$\varphi([q_i, x_i]_{i=1}^n)\leqslant \sum_{i=1}^n q_i \varphi(x_i).$
Furthermore, if $\varphi$ is lower semicontinuous then $
\varphi([r_i, x_i]_{i=1}^n)\leqslant \sum_{i=1}^n r_i \varphi(x_i)$, where $r_i> 0$, $\sum_{i=1}^n r_i =1$.
\end{prop}
\begin{proof}
The first case, we present the proof of inequality above when $q_i=1/n, i=1, \ldots , n$. Namely, we now prove that 
\begin{align*}
\varphi([n^{-1}, x_i]_{i=1}^n)\leqslant n^{-1}\sum_{i=1}^n \varphi(x_i).\tag{5.2}
\end{align*}
The proof of (5.2) is by induction on $n$. If $n=2$, (5.2) holds clearly by definition of midpoint convex function. Suppose that (5.2) holds for $n=2^k$ $(k\in \N)$, we will prove that (5.2) also holds with $n=2^{k+1}$. Indeed, for $\{x_1, x_2, \ldots , x_{2^{k+1}}\}\subset K(\X)$, we obtain
\begin{align*}
\varphi\Big(\big[2^{-(k+1)}, x_i\big]_{i=1}^{2^{k+1}}\Big)&=\varphi\Big(\Big[1/2, \big[2^{-k}, x_i\big]_{i=1}^{2^k}\,; 1/2, \big[2^{-k}, x_i\big]_{i=2^k+1}^{2^{k+1}}\Big]\Big)\leqslant \frac{1}{2}\varphi\Big(\big[2^{-k}, x_i\big]_{i=1}^{2^k}\Big)
+ \frac{1}{2}\varphi\Big(\big[2^{-k}, x_i\big]_{i=2^k+1}^{2^{k+1}}\Big)\\
&\leqslant \frac{1}{2^{k+1}}\sum_{i=1}^{2^k}\varphi(x_i)+ \frac{1}{2^{k+1}}\sum_{i=2^k+1}^{2^{k+1}}\varphi(x_i)=\frac{1}{2^{k+1}}\sum_{i=1}^{2^{k+1}}\varphi(x_i).
\end{align*}
Therefore, inequality (5.2) holds for all $n=2^k$ $(k\in \N)$. Moreover, when $n$ has form $2^k$, (5.2) holds not only for $\{x_i\} \subset K(\X)$ but also for $\{x_i\}\subset \X$.  The next step, we will prove that if (5.2) is satisfied for $n>2$ then it is also satisfied for $n-1$. Now let $\{x_1, x_2, \ldots, x_{n-1}\}\subset K(\X)$ and denote $x_n=[(n-1)^{-1}, x_i]_{i=1}^{n-1}\in K(\X)$, it follows from properties (CC.i), (CC.ii), (2.5) and induction hypothesis that
\begin{align*}
\varphi([(n-1)^{-1}, x_i]_{i=1}^{n-1})&=\varphi\left(\left[n^{-1}, x_1 ; n^{-1}, x_2 ; \ldots ; n^{-1}, x_{n-1} ; n^{-1}, \left[(n-1)^{-1}, x_i\right]_{i=1}^{n-1}\right]\right)\\
&=\varphi([n^{-1}, x_i]_{i=1}^n) \leqslant \frac{1}{n}\sum_{i=1}^n\varphi(x_i)= \frac{1}{n}\sum_{i=1}^{n-1}\varphi(x_i)+\frac{1}{n}\varphi(x_n)\\
&=\frac{1}{n}\sum_{i=1}^{n-1}\varphi(x_i)+\frac{1}{n}\varphi([(n-1)^{-1}, x_i]_{i=1}^{n-1}).
\end{align*}
This implies that
$$\varphi([(n-1)^{-1}, x_i]_{i=1}^{n-1})\leqslant \frac{1}{n-1}\sum_{i=1}^{n-1}\varphi(x_i).$$
The second case, when each $q_i$ is rational, it can be expressed as $q_i=k_i/m$, where $m, k_i$ are natural numbers for all $i=1, \ldots, n$. Then, we have
\begin{align*}
\varphi([q_i, x_i]_{i=1}^n)&=\varphi([k_i/m, x_i]_{i=1}^n)\\
&=\varphi\big([\underbrace{m^{-1}, x_1 ; \ldots ; m^{-1}, x_1}_{k_1 \mbox{ times}} ; \ldots ; \underbrace{m^{-1}, x_n ; \ldots ; m^{-1}, x_n}_{k_n\mbox{ times}}]\big)\;\mbox{ (by (2.5))}\\
&\leqslant \frac{k_1}{m}\varphi(x_1)+\cdots+ \frac{k_n}{m}\varphi(x_n)=\sum_{i=1}^n q_i \varphi(x_i)\;\mbox{ (by (5.2))}.
\end{align*}
For the remaining conclusion, when $\varphi$ is lower continuous and $r_i>0$. Then, each positive real number $r_i$ is the limit of some sequence of positive and increasing rational numbers $\{q_{ij}\}_{j=1}^\infty$. Thus, by the continuity of convex combination operation, we obtain
\begin{align*}
\varphi([r_i, x_i]_{i=1}^n)&=\varphi\big(\lim_{j\to \infty}[q_{1j}, x_1 ; \ldots; q_{nj}, x_n ; 1-(q_{1j}+\cdots+q_{nj}), a]\big)\;\mbox{ (for some }a\in K(\X))\\
&\leqslant \liminf_{j\to \infty} \varphi([q_{1j}, x_1 ; \ldots; q_{nj}, x_n ; 1-(q_{1j}+\cdots+q_{nj}), a])\\
&\leqslant \liminf_{j\to \infty}\left(q_{1j}\varphi(x_1)+\cdots+q_{nj}\varphi(x_n)+(1-(q_{1j}+\cdots+q_{nj}))\varphi(a)\right)\;\mbox{ (by the second case)}\\
&=\lim_{j\to \infty}\left(q_{1j}\varphi(x_1)+\cdots+q_{nj}\varphi(x_n)+(1-(q_{1j}+\cdots+q_{nj}))\varphi(a)\right)\\
&=r_1\varphi(x_1)+\cdots+r_n\varphi(x_n).
\end{align*}
Combining above arguments, the proposition is proved.
\end{proof}

Since the embedding theorem is only available for convexifiable domain $K(\X)$ while initial conditions are usually imposed on CC space $\X$, it is necessary to estimate quantities in $\X$ with themselves in $K(\X)$ after affecting convexification operation. The following proposition is such a result.
\begin{prop}
Let $\mathcal K$ be a compact subset of $\X$ and $\{x_n, n\geqslant 1\} \subset \mathcal K$. Then, $d\big([n^{-1}, x_i]_{i=1}^n, [n^{-1}, Kx_i]_{i=1}^n\big)\to 0$ as $n \to \infty$.
\end{prop}

\begin{proof}
For $\varepsilon >0$ arbitrarily, there exists a  finite collection $\{t_1,\ldots,t_m\}$ of elements of $\mathcal K$ such that $\mathcal K \subset \cup_{i=1}^m B(t_i, \varepsilon)$, where $B(u, r)=\{x\in \X : d(u, x)<r\}$. Denote $A_1=\mathcal K \cap B(t_1, \varepsilon), \ldots, A_l =\mathcal K \cap B(t_l,\varepsilon)\cap\big\{\cup_{k=1}^{l-1}B(t_k,\varepsilon)\big\}^c$, $l=2,\ldots, m$. For each $n$, let us define $y_n=t_l$ if $x_n\in A_l$, so $d(x_n, y_n)<\varepsilon$ for all $n$. By triangular inequality and (2.6),
\begin{align*}
d\big([n^{-1}, x_i]_{i=1}^n, [n^{-1}, Kx_i]_{i=1}^n\big) &\leqslant d\big([n^{-1}, x_i]_{i=1}^n, [n^{-1}, y_i]_{i=1}^n\big)+ d\big([n^{-1}, y_i]_{i=1}^n, [n^{-1}, Ky_i]_{i=1}^n\big)\\
&\;\;\;\;\;+ d\big([n^{-1}, Ky_i]_{i=1}^n, [n^{-1}, Kx_i]_{i=1}^n\big)\\
&\leqslant 2 n^{-1} \sum_{i=1}^n d(x_i, y_i)+ d\big([n^{-1}, y_i]_{i=1}^n, [n^{-1}, Ky_i]_{i=1}^n\big)\leqslant 2\varepsilon + (I_n).
\end{align*} 
We now show that $(I_n)=d\big([n^{-1}, y_i]_{i=1}^n, [n^{-1}, Ky_i]_{i=1}^n\big) \to 0$ as $n\to \infty$. For each $l=1, \ldots, m$, put 
\begin{align*}
z_{l,n}=\mbox{card}\{1\leqslant i \leqslant n : y_i =t_l\},\; \mbox{ and } \mathcal T_n=\{l : 1\leqslant l \leqslant m, z_{l, n}>0\},\;n\geqslant 1.
\end{align*}
Then, $\{z_{l,n}, n\geqslant 1\}$ is the non-decreasing sequence for each $l$. 
By (CC.i) and property (2.1), we obtain
  \begin{align*}
[n^{-1}, y_i]_{i=1}^n=\big[n^{-1}z_{l, n},\big[z_{l, n}^{-1}, t_l\big]_{i=1}^{z_{l, n}}\big]_{l\in \mathcal T_n} \;\mbox{ and }\;[n^{-1}, Ky_i]_{i=1}^n=\big[n^{-1}z_{l, n},\big[z_{l, n}^{-1}, Kt_l\big]_{i=1}^{z_{l, n}}\big]_{j\in \mathcal T_n}=\big[n^{-1}z_{l, n}, Kt_l\big]_{l\in \mathcal T_n}.
 \end{align*}
For each $l=1, \ldots, m$, we have
$\lim_{n\to \infty}d([n^{-1}, t_l]_{i=1}^n, Kt_l)= 0$ by the definition of $K$. Thus, there exists $n_{\varepsilon, m}\in \mathbb N$ such that for all $n\geqslant n_{\varepsilon, m}$ and for all $l=1,\ldots, m,$
\begin{align*}
d([n^{-1}, t_l]_{i=1}^n, Kt_l)<\frac{\varepsilon}{m}.\tag{5.3}
\end{align*}
We put
$$N_{l,\varepsilon, m}=\max_{1\leqslant k < n_{\varepsilon, m}}d\big([k^{-1}, t_l]_{i=1}^k, Kt_l\big),\;N_{\varepsilon, m}=\max_{1\leqslant l \leqslant m}N_{l,\varepsilon, m}$$
and choose the smallest integer number $n'_{\varepsilon, m}$ such that $n'_{\varepsilon, m}\geqslant \varepsilon^{-1} m.N_{\varepsilon, m}.n_{\varepsilon,m}.$ Now, for $n\geqslant n'_{\varepsilon, m}$:\\
- If $z_{l, n} \geqslant n_{\varepsilon, m}$, then it follows from (5.3) and $n^{-1}z_{l, n}\leqslant 1$ that

$$\frac{z_{l, n}}{n}d\Big(\big[z_{l, n}^{-1},t_l\big]_{i=1}^{z_{l, n}}, Kt_l\Big)<\frac{\varepsilon}{m}.$$
- If $0<z_{l, n}< n_{\varepsilon, m}$, then
\begin{align*}
\frac{z_{l, n}}{n}d\Big(\big[z_{l, n}^{-1}, t_l\big]_{i=1}^{z_{l, n}}, Kt_l\Big)< \frac{n_{\varepsilon, m}}{n'_{\varepsilon, m}}.N_{\varepsilon, m}\leqslant \frac{\varepsilon}{m}.
\end{align*}
Hence, for all $n\geqslant n'_{\varepsilon, m}$
$$\frac{z_{l, n}}{n}d\Big(\big[z_{l, n}^{-1}, t_l\big]_{i=1}^{z_{l, n}},Kt_l\Big)\leqslant \frac{\varepsilon}{m}.$$
This implies that
\begin{align*}
(I_n)\leqslant \sum_{l\in \mathcal T_n}\frac{z_{l, n}}{n}d\Big(\big[z_{l, n}^{-1}, t_l\big]_{i=1}^{z_{l, n}},Kt_l\Big)\leqslant \varepsilon
\end{align*}
for all $n\geqslant n'_{\varepsilon, m}$, so $(I_n)\to 0$ as $n\to \infty$. The proof is completed. 
\end{proof}



\begin{center}
\textbf{Acknowledgements}
\end{center}

The author wishes to thank Assoc. Prof. Pedro Ter\'{a}n (Escuela Polit\'ecnica de Ingenier\'ia, Departamento de Estad\'istica e I.O. y D.M., Universidad de Oviedo, E-33071 Gij\'on, Spain) for helpful discussions.
The author also would like to thank Dr. Tobias Fritz (Perimeter Institute for Theoretical Physics, Waterloo ON, Canada) for letting us know the reference \cite{CF}.

\end{document}